\documentclass{amsart}

\newtheorem{theorem}{Theorem}[section]
\newtheorem{lemma}[theorem]{Lemma}
\newtheorem{cor}[theorem]{Corollary}
\newtheorem{Cor}[theorem]{Corollary}

\theoremstyle{definition}

\newtheorem{definition}[theorem]{Definition}

\newtheorem{example}[theorem]{Example}

\theoremstyle{remark}

\usepackage{latexsym,amsfonts}
\numberwithin{equation}{section}

\def\co{\colon\thinspace}

\newcommand{\nser}{\mathcal{N}} 

\newcommand{\cser}{\mathcal{C}}
 
\DeclareMathOperator{\ad}{ad}
\DeclareMathOperator{\Ad}{Ad}
\DeclareMathOperator{\Biad}{Biad}
\DeclareMathOperator{\Bider}{Bider}
\DeclareMathOperator{\Cloder}{Cloder}

\DeclareMathOperator{\codim}{codim}
\DeclareMathOperator{\soc}{Soc}
\DeclareMathOperator{\core}{Core}
\DeclareMathOperator{\ch}{char}

\DeclareMathOperator{\Hom}{Hom}
\DeclareMathOperator{\alg}{alg\langle}

\DeclareMathOperator{\Leib}{Leib}
\DeclareMathOperator{\sym}{{}^\text{sym}\mspace{-4mu}}
\DeclareMathOperator{\asym}{{}^\text{asym}\mspace{-4mu}}

\newcommand{\sn}{\mbox{$\triangleleft\mspace{-1.8mu}\triangleleft\medspace$}}
\newcommand{\snr}{\mbox{$\triangleleft\mspace{-1.8mu}\triangleleft_r\medspace$}}

\newcommand{\f}{\mathfrak}
\newcommand{\fH}{{\mathfrak H}}

\newcommand{\fF}{{\mathfrak F}}
\newcommand{\fK}{{\mathfrak K}}
\newcommand{\fN}{{\mathfrak N}}

\newcommand{\fU}{{\mathfrak U}}
\newcommand{\fCs}{{\mathfrak{Cs}}}
\DeclareMathOperator{\id}{\lhd}
\DeclareMathOperator{\idr}{\lhd_r}
\DeclareMathOperator{\ideq}{\unlhd}
\newcommand{\Lie}{{}^{\text{Lie}} \mspace{-1.5mu}{}}

\DeclareMathOperator{\ev}{Ev}
\DeclareMathOperator{\loc}{Loc}

\DeclareMathOperator{\quot}{\textsc{q}\mspace{-2mu}}
\DeclareMathOperator{\sdir}{\textsc{r}\mspace{-2mu}}
\DeclareMathOperator{\frat}{{\textsc{e}}_\Phi\mspace{-3mu}}
\DeclareMathOperator{\fratn}{{\textsc{e}}_\Phi^{\id}\mspace{-3mu}}
\DeclareMathOperator{\fratsn}{{\textsc{e}}_\Phi^{\sn}\mspace{-3mu}}
\DeclareMathOperator{\fratsnr}{{\textsc{e}}_\Phi^{\snr}\mspace{-3mu}}
\DeclareMathOperator{\prim}{\textsc{p}\mspace{-2mu}}
\DeclareMathOperator{\Proj}{Proj}
\DeclareMathOperator{\Cov}{Cov}

\begin{document}
\title{Schunck classes of soluble Leibniz algebras}
\author{Donald W. Barnes}
\address{1 Little Wonga Rd, Cremorne NSW 2090 Australia}
\email{donwb@iprimus.com.au}
\thanks{This work was done while the author was an Honorary Associate of the
School of Mathematics and Statistics, University of Sydney.}

\begin{abstract}  I set out the theory of Schunck classes and projectors for soluble Leibniz algebras, parallel to that for Lie algebras.  Primitive Leibniz algebras come in pairs, one (Lie) symmetric, the other antisymmetric.  A Schunck formation containing one member of a pair also contains the other.  If $\fH$ is a Schunck formation and $H$ is an $\fH$-projector of the Leibniz algebra $L$, then $H$ is intravariant in $L$.  An example is given to show that the assumption that the Schunck class $\f{H}$ is a formation cannot be omitted.
\end{abstract}

\maketitle
\section{Introduction} \label{sec-intro}
The theory of Schunck classes, formations and projectors was originally developed for finite soluble groups.  This theory is set out in Doerk and Hawkes \cite{DH}.  It was adapted to Lie algebras in Barnes and Gastineau-Hills \cite{BGH} (but set out in older terminology), and to restricted Lie algebras in Barnes \cite{Res}.  Most of the theory for Leibniz algebras is a straightforward translation of the theory for Lie algebras.  Where proofs are the same as for Lie algebras, I omit the proofs, indicating this by placing the end of proof symbol \qedsymbol\  at the end of the statement of the result.

 A (left) Leibniz algebra is an algebra $L$ over a field $F$ for which all the left multiplications are derivations, that is,
$$a(bc) = (ab)c + b(ac)$$
for all $a,b,c \in L$.   The basic properties of Leibniz algebras and their modules may be found in Ayupov and Omirov \cite{AyO} or Patsourakos \cite{Pats}.  Some further definitions and notations used in this paper are in Barnes \cite{LeibEngel}, along with some lemmas fundamental to the study of Schunck classes.

That the left ideal generated by the squares of elements of a Leibniz algebra $L$ is a two-sided ideal whose quotient is a Lie algebra is well known.  I propose to call this ideal the Leibniz kernel of $L$ and to denote it by $\Leib(L)$.
The following stronger result appears not to have been noticed.

\begin{lemma}   Let $L$ be a Leibniz algebra.   The subspace $\langle x^2 \mid x \in L\rangle$ is a 2-sided ideal of $L$. 
\end{lemma}

\begin{proof} We have $x(xy) = (xx)y + x(xy)$, so $(xx)y = 0$.  Also
$$(x+y^2)^2 = x^2 + xy^2 + y^2x+ y^2 y^2 = x^2 + x (y^2).$$
Thus $xy^2 = (x+y^2)^2 - x^2 \in \langle a^2 \mid a \in L \rangle$
\end{proof}

The powers $L^n$ of a Leibniz algebra $L$ are defined inductively by $L^1 = L$ and $L^{n+1} = LL^n$.  I quote Ayupov and Omirov \cite[Lemma 1]{AyO}.

\begin{lemma} Let $L$ be a Leibniz algebra.  Then $L^rL^s \subseteq L^{r+s}$.
\end{lemma}

\begin{proof}  We use induction over $r$.  The result holds by definition for $r=1$.  For $r>1$, we have 
$L^rL^s = (LL^{r-1})L^s \subseteq L(L^{r-1}L^s) + L^{r-1}(LL^s) $.  But $L^{r-1}L^s \subseteq L^{r+s-1}$ and  $ L^{r-1}L^{s+1} \subseteq L^{r+s}.$  The result follows.
\end{proof}

\begin{cor} \label{powers} Suppose $A \ideq L$.  Then $A^r \ideq L$.
\end{cor}

\begin{proof}  Since left multiplications are derivations, $A^r$ is a left ideal.  I use induction over $r$ to prove it a right ideal.  We have \begin{equation*}
\qquad\qquad(AA^{r-1})L \subseteq A(A^{r-1}L) + A^{r-1}(AL) \subseteq AA^{r-1}+A^{r-1}A\subseteq A^r. \qquad\quad\qed \end{equation*}
\renewcommand{\qed}{}
\end{proof}

It might seem superfluous to prove Corollary \ref{powers}, but  a product of ideals of a Leibniz algebra need not be an ideal.

\begin{example}  Let $N = \langle a,b,c,d \rangle$ with multiplication $ab=c, ba=d, a^2 =b^2 =c^2 =d^2 =0$ and $c,d$ central.  Let $L = \langle x, N\rangle$ with $xa=a=-ax, xb=bx=0,xc=c, xd=d, cx=d, dx=-d$ and $x^2 = 0$. Let $A= \langle a,c,d\rangle$ and $B = \langle b,c,d \rangle$.  Then $A,B$ are ideals of $L$, but $AB = \langle c\rangle$ which is not an ideal.
\end{example}

Bosko, Hird, McAlister, Schwartz and Stagg, \cite{BHMSS} have proved (as a corollary to a deeper result) that  the sum of two nilpotent ideals of a finite-dimensional Leibniz algebra is nilpotent.  Using Corollary \ref{powers}, this can be proved without the assumption of finite dimension.

\begin{lemma}  \label{nilsum}Suppose that the ideals $N_1,N_2$ of $L$ are nilpotent of classes $c_1$ and $c_2$.  Then $N_1+N_2$ is nilpotent of class at most $c_1 + c_2$.
\end{lemma}

\begin{proof}  If either of $c_1, c_2$ is $0$, the result holds trivially.  I use induction over $c_1+c_2$.  Put $N = N_1+N_2$.  Then $N/N_1^{c_1} = N_1/N_1^{c_1} + (N_2 + N_1^{c_1})/N_1^{c_1}$ is nilpotent of class at most $c_1 +c_2 - 1$.  Thus $N^{c_1+c_2} \subseteq N_1^{c_1}$ and $N_1N^{c_1+c_2} = 0$.  Similarly, $N_2N^{c_1+c_2} =0$.  Thus $N^{c_1+c_2 +1} = 0$.
\end{proof}

For a finite-dimensional Leibniz algebra $L$, the nil radical $N(L)$ is defined to be the sum of all nilpotent ideals of $L$.  By Lemma \ref{nilsum}, $N(L)$ is nilpotent.

For an $L$-(bi)module $M$ and the representation $(S,T)$ of $L$ on $M$, I write $\ker(M)$ or $\ker(S,T)$ for $\ker(S) \cap \ker(T)$.  The following result is implicit in Loday and Pirashvili \cite{LP} and proved in Barnes \cite{LeibSub}.

\begin{lemma}\label{irred}  Let $L$ be a finite-dimensional Leibniz algebra and let $M$ be a finite-dimensional irreducible $L$-module.    Then $L/\ker(M)$ is a Lie algebra and either $ML=0$ or $mx = -xm$ for all $x \in L$ and $m \in M$. \end{lemma}

An $L$-module $M$ satisfying $mx=-xm$ for all $x \in L$ and all $m \in M$ is called (Lie) symmetric, while a module satisfying $ML=0$ is called antisymmetric.  For a given module $M$, I denote by $\sym{M}$, the module with the same left action as $M$ but with the right action replaced by the symmetric action.  I denote by $\asym M$ the module $M$ with the right action replaced with the zero action.  It follows from Lemma \ref{irred}, that primitive algebras come in pairs, one in which the action on the socle is symmetric, the other antisymmetric.  For the primitive algebra $P$ with socle $\soc(P)=C$, I denote by $\sym P$ the primitive algebra which is the split extension of $\sym C$ by $P/C$ and by $\asym P$ the split extension of $\asym C$ by $P/C$.  Necessarily, $P$ is one or other of these.  The $1$-dimensional algebra is primitive, both symmetric and antisymmetric, and is its own pair.

I want to show that in any module $V$, there is a composition series in which the symmetric factors are above the antisymmetric factors.  This is complicated by the fact that  a module with trivial action is both symmetric and antisymmetric. To cope with this, I give the following definition.

\begin{definition}  Let $V$ be an $L$-module and let $A/B$ be a composition factor of $V$.  Let $X$ be the split extension of $V$ by $L$.  We say that $A/B$ is {\em inner} relative to $V$ if $A/B \subseteq \Leib(X/B)$ and that $A/B$ is {\em outer} otherwise.
\end{definition}

Let $K = \Leib(X)$.  Since $\Leib(X/B) = \Leib(X)+B/B$, $A/B$ inner is equivalent to $K+B \supseteq A$.  As $K+B \supseteq A$ if and only if there exists $k \in K \cap A$, $k \notin  B$, $A/B$ outer is equivalent to $K \cap A \subseteq B$.  Note that the property is {\em relative to} $V$, not an inherent property of the module $A/B$.

\begin{lemma} \label{asyminLeib} Let $L$ be a Leibniz algebra and let $V$ be a non-trivial antisymmetric irreducible $L$-module.  Let $X$ be the split extension of $V$ by $L$.  Then $V \subseteq \Leib(X)$.
\end{lemma}

\begin{proof}  If $v \in V$ and $v\ne 0$, then there exists $x \in L$ such that $xv \ne 0$.  But $xv = (x+v)^2 - x^2 \in \Leib(X)$.  Thus $V \cap \Leib(X) \ne 0$.  Since $V$ is irreducible, this implies that $V  \subseteq \Leib(X)$.
\end{proof}

Thus a non-trivial antisymmetric composition factor of any module is inner.  Clearly, a non-trivial symmetric composition factor is outer.  For trivial composition factors, the usual correspondence between factors of different composition series need not preserve the property of being inner.

\begin{example}  Let $L = \langle x \rangle$ and let $V = \langle a,b,c \rangle$ with $xa = b, xb = xc = ax =bx =cx =0$.  Then in the notation used above, $\Leib(X)= K = \langle b \rangle$ and we have (at least) the following submodules:
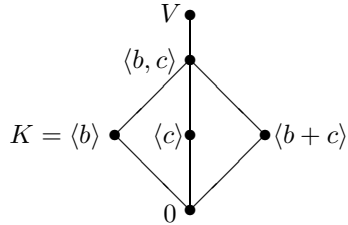
\begin {center} \setlength{\unitlength}{1mm}
\begin{picture}(50.5,30) 
\put(25,28){\circle*{1.5}}
\put(21,27){$V$}
\put(25,28){\line(0,-1){6}} \put(25,22){\circle*{1.5}} \put(16,21){$\langle b,c \rangle$} 
\put(25,22){\line(-1,-1){10}} \put(15,12){\circle*{1.5}} \put(1,11){$K = \langle b \rangle$} 
\put(20,11){$\langle c \rangle$} \put(25,12){\circle*{1.5}} \put(25,22){\line(0,-1){10}} 
\put(25,22){\line(1,-1){10}} \put(35,12){\circle*{1.5}} \put(36,11){$\langle b+c \rangle$}
\put(15,12){\line(1,-1){10}} \put(25,2) {\circle*{1.5}} \put(21.5,0.5){$0$}
\put(25,12){\line(0,-1){10}} \put(35,12){\line(-1,-1){10}}
\end{picture}
\end{center}
The factor $\langle b,c\rangle/\langle c \rangle$ of the composition series $V \supset \langle b,c\rangle \supset \langle c \rangle \supset 0$ is inner and corresponds to the factor $\langle b+c \rangle$ of the composition series $V \supset \langle b,c\rangle \supset \langle b+c \rangle \supset 0$.  But $\langle b+c \rangle$ is outer.
\end{example}

\begin{theorem}\label{inout} Let $L$ be a Leibniz algebra and let $V$ be an $L$-module.  Then there exists a unique submodule $W$ of $V$ such that every composition factor of $W$ is inner and every composition factor of $V/W$ is outer. 
\end{theorem}

\begin{proof} Let $X$ be the split extension of $V$ by $L$ and let $K = \Leib(X)$. Let $W$ be a submodule of $V$ and let 
$$V = V_0 \supset V_1 \supset \dots \supset V_r = W = W_0 \supset W_1 \supset \dots \supset W_s = 0$$
with all the $V_i/V_{i+1}$ outer and all the $W_j/W_{j+1}$ inner.  Since $K \cap V_i \subseteq V_{i+1}$, $K \cap V \subseteq W$.  Since $K + W_i \supseteq W_{i-1}$, $K \cap V \supseteq W$.  Thus $W = K \cap V$.  Clearly, $W = K \cap V$ has the asserted properties.
\end{proof}

We cannot interchange inner and outer in Theorem \ref{inout}.  To see this, we use the Loday-Pirashvili short exact sequence.  Let $L$ be a Lie algebra and let $V$ be a left $L$-module.  We put $\bar{V} = \Hom(L,V) \oplus V$ as left $L$-module and make it a bimodule by setting $(f,v)x = (-xf, f(x))$ for $x\in L$, $f \in \Hom(L,V)$ and $v \in V$.  The Loday-Pirashvili short exact sequence is the sequence
$$0 \to \asym{V} \to \bar{V} \to \Hom(L,V) \to 0.$$

\begin{lemma} \label{LP}  Suppose that there exist linearly independent elements $a,b \in L$ such that the left action $T_a$ of $a$ on $V$ is invertible.  Then the Loday-Pirashvili short exact sequence does not split.
\end{lemma}

\begin{proof}  Suppose that $U$ is a submodule which complements $\asym{V}$ in $\bar{V}$.Then for some linear map $\phi \co \Hom(L,V) \to V$, we have $U = \{(f,\phi(f)) \mid f \in \Hom(L,V) \}$.  For $U$ to be a submodule, for all $x \in L$ and $f \in \Hom(L,V)$, we need
$$x(f, \phi(f)) = (xf,x\phi(f)) = (xf, \phi(xf)),$$
and
$$(f,\phi(f))x =(-xf, f(x)) =(-xf, \phi(-xf)).$$
  Thus we must have $\phi(xf) = x\phi(x)$ and $f(x)=-\phi(xf) = -T_x\phi(f)$.  Since $a,b$ are not linearly dependent, there exists $f \in \Hom(L,V)$ such that $f(a)=0$ and $f(b)\ne 0$.  Then $\phi(f) = - T_a^{-1}f(a) = 0$ and $f(b) = -T_b \phi(f) = 0$ contrary to the choice of $f$.  Hence the sequence  does not split.
\end{proof}

 As for soluble Lie algebras, we have:
\begin{lemma} \label{nrad} Let $L$ be a soluble Leibniz algebra. Then the intersection of the centralisers of the chief factors of $L$ is $N(L)$. \qed
\end{lemma}

For a Lie algebra $L$, the Frattini subalgebra $\Phi(L)$ is an ideal if $\ch(F)=0$ by the Tuck-Towers Theorem \cite[Corollary 3.3]{Tow}, or if $L$ is soluble by Barnes and Gastineau-Hills \cite[Lemma 3.4]{BGH}.  The latter result fails for Leibniz algebras as is shown by the following example.

\begin{example}  Suppose $\ch(F)=p$.  Let $M = \langle a,b,z \mid ab=z, az=bz=0 \rangle$ be the non-abelian $3$-dimensional nilpotent Lie algebra over $F$ and let $V$ be the faithful irreducible $M$-module $V = \langle v_0, \dots, v_{p-1} \rangle$ with $av_i = iv_{i-1}, bv_i = v_{i+1}$ and $zv_i=v_i$.  We make this a bimodule setting $VM=0$.  Let $L$ be the split extension of $V$ by $M$.  Then $L$ is a primitive Leibniz algebra, $M$ is the only maximal subalgebra of $L$ not containing $V$ and $\Phi(L) = \langle z \rangle$ which is not an ideal of $L$.
\end{example}

Because of this, it is sometimes convenient to work instead with the Frattini ideal $\Psi(L)$, the largest ideal of $L$ contained in $\Phi(L)$.  However, the characteristic $0$ result does extend to Leibniz algebras.

\begin{theorem}  Let $L$ be a Leibniz algebra over the field $F$ of characteristic $0$.  Then $\Phi(L) \id L$.
\end{theorem}

\begin{proof}  If $A \id L$, by induction, we may suppose that the intersection $K$ of the maximal subalgebras containing $A$ is an ideal.  The result holds if $L$ is a Lie algebra, so we may suppose that there exists an abelian minimal ideal $A$ of $L$.   If every maximal subalgebra contains some minimal ideal, then the result holds, so we may suppose that there exists a maximal subalgebra $M$ which does not contain any minimal ideal of $L$.  Then $M \cap A =0$.  Since $\cser_L(A) \cap M \id L$, we have $\cser_L(A)=A$ and by Lemma \ref{irred} , $L/A$ is a Lie algebra.  

Consider the Lie algebra $X = L/A$.  Let $R$ be the soluble radical of $X$.  Let $B$ be a minimal $R$-submodule of $A$.  Then $R'$ acts trivially on $B$ since $R'+A$ is nilpotent.  But all $R$-composition factors of $A$ are isomorphic, so $R'$ acts nilpotently on $A$.  Therefore, there exists $a \in A, a \ne 0$ such that $R'a=0$.  As $R'\id X$, $\{a \in A \mid R'a=0\}$ is an $X$-submodule of $A$.  Since $A$ is a faithful irreducible $X$-module, this implies that $R'=0$.

Let $S$ be a Levi factor of $X$.  Then $\Phi(S)=0$ by the Tuck-Towers Theorem \cite[Corollary 3.3]{Tow}.  Since $X/R \simeq S$, $\Phi(X) \subseteq R$.  As $S$-module, $R$ is completely reducible, $R = R_1 \oplus \dots \oplus R_n$ for some irreducible $S$-submodules $R_i$.  Put $R^i = \sum_{j \ne i}R_i$.  Then $S+R^i$ is a maximal subalgebra of $X$ and it follows that $\Phi(X) = 0$. Thus the intersection of the maximal subalgebras of $L$ which contain $A$ is $A$ and it follows that $\Phi(L) = 0$.
\end{proof}

\section{Classes and Projectors}\label{proj}
Following the notations of Doerk and Hawkes, I denote the class of all finite dimensional soluble Leibniz algebras over the field $F$ by $\f{S}$, the class of nilpotent Leibniz algebras by $\f{N}$, the class of abelian Leibniz algebras by $\f{A}$ and the class of primitive Leibniz algebras by $\f{P}$.  For a class
$\f{X}$, I define the closure operations
\begin{equation*}\begin{split}
\quot \f{X} &= \{L/K \mid L \in \f{X}\}\\
\sdir \f{X} &= \{L \mid \exists \text{ ideals $K_i$ of $L$  with $(L/K_i) \in \f{X}$ and $\cap_i K_i = 0$}\}\\
\frat \f{X} &= \{L \mid \exists \text{ ideal $K$ of $L$  with $K \le \Phi(L)$ and $L/K \in \f{X}$}\}\\
\prim \f{X} &= \{L \mid \quot(L) \cap \f{P} \subseteq \f{X}\}.\\
\end{split}\end{equation*}
Thus $\quot\f{X}$ is the class of quotients of Leibniz algebras in $\f{X}$, $\sdir\f{X}$ is the  
class of subdirect sums and $\frat\f{X}$ the class of Frattini extensions of algebras in $\f{X}$, while
$\prim\f{X}$ is the class of all algebras whose primitive quotients are in $\f{X}$.

\begin{definition} A non-empty class $\f{X}$ of soluble Leibniz algebras which is $\quot$-closed,
that  is,   $\quot\f{X} =\f{X}$, is called a {\em homomorph}.  An $\sdir$-closed homomorph is
called a {\em formation.}  A non-empty class which is $\frat$-closed is called {\em saturated.}     
A non-empty class $\f{X}$ satisfying $\prim\f{X} = \f{X}$ is called a {\em Schunck class.}
\end{definition}

These definitions differ from those of Doerk and Hawkes by the inclusion of the requirement, 
convenient for the  theory of Leibniz algebras but not for that of finite groups,  that the classes    
be non-empty.  Note also that {\em saturation} had a different meaning, explained below, in the older
terminology.  Clearly, a Schunck class is a saturated homomorph.  If $\f{X}$ is a homomorph, then $\prim\f{X}$ is the smallest Schunck class containing $\f{X}$.

There are some results in the literature giving stronger forms of $\frat$-closure.  We say that $\f{X}$ is $\fratn$-closed ($\fratsn$-closed, $\fratsnr$-closed) if $A \id L$, (respectively $A \sn L$, $A \snr L$), $B \id A$, $A/B \in \f{X}$ and $B \subseteq \Phi(L)$ imply $A \in \f{X}$.   (It is not claimed that $\fratn, \fratsn$ or $\fratsnr$ are closure operations.) It was shown in Barnes \cite[Theorem 3.6]{LeibEngel} that the class $\fN$ of nilpotent Leibniz algebras is $\fratsnr$-closed.   In Barnes and Newell \cite[Theorem 4.3]{BN}, it was proved that all  Schunck classes of soluble Lie algebras are $\fratsn$-closed.  This does not hold for Schunck classes of Leibniz algebras.  We prove below, that Schunck formations of Leibniz algebras are $\fratsn$-closed, but give an example of a Schunck class which is not $\fratn$-closed.

\begin{lemma}  Let $\f{X}$ be  homomorph which is not a formation.  Then there exists a Leibniz algebra $L$ with minimal ideals, $K_1, K_2$ such that $L/K_i \in \f{X}$, $i=1,2,$ but
$L \not\in \f{X}$. \qed\end{lemma}

\begin{definition}  Let $\f{X}$ be a class of Leibniz algebras.  A  subalgebra $U$ of $L$ is
called {\em $\f{X}$-maximal} in $L$ if it is maximal in the set of those subalgebras of $L$
which are in $\f{X}$.
\end{definition}

\begin{definition} Let $\f{X}$ be a homomorph.  A subalgebra $U$ of $L$ is called an
{\em $\f{X}$-projector} of $L$ if, for every ideal $K$ of $L$, $U+K/K$ is $\f{X}$-maximal
in $L/K$. \end{definition}

\begin{definition} Let $\f{X}$ be a  homomorph.  A subalgebra $U$ of $L$ is called an {\em $\f{X}$-covering subalgebra} of $L$ if, whenever $V$ is a  subalgebra of $L$
containing $U$ and $K$ is an ideal of $V$ with $V/K \in \f{X}$, we have $U+K = V$.  (In the older terminology of \cite{BGH}, these subalgebras were called $\f{X}$-projectors.)
\end{definition}

Thus, an $\f{X}$-covering subalgebra $U$ of $L$ is an $\f{X}$-projector of every  subalgebra of $L$ which contains $U$.  We denote the (possibly empty) set of $\f{X}$-projectors of $L$ by $\Proj_\f{X}(L)$ and the set of $\f{X}$-covering subalgebras by $\Cov_\f{X}(L)$.  

\begin{lemma}\label{Hcomp} Let $\f{X}$ be a homomorph and let $A$ be a minimal ideal of $L$.  Suppose $L/A \in \f{X}$, $L \not\in \f{X}$ and that $U \in \Proj_\f{X}(L)$.  Then $U$ complements $A$ in $L$. \qed\end{lemma}

\begin{lemma} \label{selfnorm} Let $\f{X} \ne \{0\}$ be a homomorph.  Suppose $H \in \Cov_\f{X}(L)$ and that $H$ is contained in the subalgebra $U$ of $L$.  Then $\nser^r_L(U) = U$. \end{lemma}
\begin{proof}  I prove first that $\nser_L(U)=U$.  Now $N = \nser_L(U)$ is a subalgebra of $L$ and $U \ideq N$.  If $N \ne U$, there exists a subalgebra $K$ of $N$ such that $K/U$ is a $1$-dimensional algebra.  As $K/U \in \f{X}$, $H+U = K$.  But $H \subseteq U$, so $K \subseteq U$ contrary to the choice of $K$.

I now use induction over $\dim(L)$.  The result holds trivially if $\dim(L)=1$.  Let $A$ be a minimal ideal of $L$.  Suppose $U+A < L$.  If $x \in \nser^r_L(U)$, then $x+A \in \nser^r_{L/A}(U+A/A) = U+A/A$ and $x \in U+A$.  But $\nser^r_{U+A}(U) = U$, so $x \in U$.  Hence we may suppose that $U+A = L$ and $U$ complements $A$ in $L$.  As $A$ is irreducible as $U$-module, it is either symmetric or antisymmetric.  As $U$-module, $L = U \oplus A$.  In either case, we have $\nser^r_L(U) = \nser_L(U) = U$. 
\end{proof}

\begin{lemma}\label{primcov} Let $\f{X}$ be a  homomorph and let $L$ be a primitive algebra
not in $\f{X}$ but with $L/\soc(L) \in \f{X}$.  Then $\Cov_\f{X}(L) = \Proj_\f{X}(L)$ and is the set of all
complements to $\soc(L)$ in $L$. \qed\end{lemma}

Let $A$ be an abelian ideal of $L$ and let $\lambda_a\co L \to L$ be the left multiplication map $\lambda_a(x) = ax$.  Then the map $\alpha_a = 1 + \lambda_a \co L \to L$ is an automorphism of $L$.

\begin{lemma} \label{conj} Let $\f{X}$ be a homomorph. 
Let $A$ be an abelian  ideal of $L$. Suppose $L/A \in \f{X}$ and that $U_1, U_2 \in \Cov_\f{X}(L)$.  Then there exists $a \in A$ such that $\alpha_a(U_1) = U_2$. \qed\end{lemma}

\begin{lemma}\label{transp} Let $\f{X}$ be a homomorph.  Let $K$ be an ideal of $L$. 
Suppose that $V/K \in \Proj_\f{X}(L/K)$ and $U \in \Proj_\f{X}(V)$.  Then $U \in \Proj_\f{X}(L)$. \qed\end{lemma}

\begin{lemma}\label{transc} Let $\f{X}$ be a  homomorph.  Let $K$ be an ideal of $L$. 
Suppose that $V/K \in \Cov_\f{X}(L/K)$ and $U \in \Cov_\f{X}(V)$.  Then $U \in \Cov_\f{X}(L)$.\qed\end{lemma}

\begin{lemma} \label{H1} Let $\f{X} \ne \{0\}$ be a homomorph. Let $A$ be a minimal ideal of $L$. Suppose $L/A \in \f{X}$, $L \not\in \f{X}$ and $\Cov_\f{X}(L) \ne \emptyset$.  Then $\Cov_\f{X}(L)$ is the set of all complements to $A$ in $L$. \qed \end{lemma}

\begin{definition} A  homomorph $\f{X}$ is called {\em projective} if, for every soluble  
Leibniz algebra $L$, we have $\Proj_\f{X}(L) \ne \emptyset$.  It is called a {\em Gasch\"utz} class
if, for every soluble Leibniz algebra $L$, $\Cov_\f{X}(L) \ne \emptyset$. \end{definition} 

In the older terminology, what are here called $\f{X}$-covering subalgebras were called
$\f{X}$-projectors, and what are here called Gasch\"utz classes were called saturated homomorphs.

\begin{theorem}\label{GPS} \newcounter{bean} Let $\f{X}$ be a homomorph.  Then the following are equivalent:
\begin{list}{}{\usecounter{bean}}
\item[\rm (a)] $\f{X}$ is a Gasch\"utz class.
\item[\rm (b)] $\f{X}$ is a projective class.
\item[\rm (c\,)] $\f{X}$ is a Schunck class. \qed\end{list}\end{theorem}

\begin{lemma}\label{closure} Let $\f{H}$ be a  homomorph.  Suppose $U$ is an $\f{H}$-covering
subalgebra of $L$.    Then $U$ is a  $\prim\f{H}$-covering subalgebra of $L$. \qed  \end{lemma}

The next lemma is a slightly modified version of Doerk and Hawkes \cite[Lemma 3.14, p.295]{DH}.
\begin{lemma} \label{nilsup} Let $\f{X}$ be a  Schunck class.  Let $N$ be a nilpotent  ideal
of $L$ and  let $U$ be an $\f{X}$-maximal  subalgebra of $L$ such that $L = U + N$.  Then $U \in \Cov_\f{X}(L)$. \end{lemma}

\begin{proof}  We use induction on the dimension of $L$.  If $L \in \f{X}$, the result holds, so suppose
$L \not\in \f{X}$.  Then there exists a  ideal $K$ of $L$ such that $P=L/K$ is primitive, not in
$\f{X}$ but with $P/\soc(P) \in \f{X}$.  Then $N \not\subseteq K$ since $L/N \simeq U/(U\cap N) \in
\f{X}$.  Hence $N+K/K$ is a non-zero nilpotent ideal of $P$ and so $N+K/K = \soc(P)$.  But $U+K/K$ is a maximal  subalgebra of $P$ complementing $N+K/K$, so $U+K/K \in \Cov_\f{X}(L/K)$ by Lemma \ref{primcov}.  Now $(U+K) \cap N$ is a nilpotent  ideal of $U+K$. Also $U+((U+K) \cap N) = U+K$ by the modular law for subspaces.  As $U+K \ne L$, by induction, we have that $U \in \Cov_\f{X}(U+K)$.  By Lemma \ref{transc}, $U \in \Cov_\f{X}(L)$. \end{proof}

\begin{theorem}\label{projcov} Let $\f{X}$ be a  Schunck class of soluble Leibniz algebras.  Suppose $U \in \Proj_\f{X}(L)$. Then $U \in \Cov_\f{X}(L)$. \qed\end{theorem}

\begin{lemma}\label{schunck} Let $\f{H}$ be a homomorph.  A necessary and sufficient condition for $\f{H}$ to be a  Schunck class is that $L \not\in \f{H}$, $A$ a minimal  ideal of $L$ and $L/A \in \f{H}$ implies $\Proj_\f{H}(L) \ne \emptyset$. \end{lemma}

\begin{proof}  By Theorem \ref{GPS}, the condition is necessary.  Suppose $\f{H}$ satisfies the condition.  We use induction over $\dim(L)$ to prove for all $L$ that $\Proj_\f{H}(L) \ne \emptyset$.  Let $A$ be a minimal ideal of $L$.  By induction, there exists a subalgebra $U \supseteq A$ such that $U/A \in \Proj_\f{H}(L/A)$.

If $U \subset L$, then by induction, there exists $H \in \Proj_\f{H}(U)$ and then $H \in \Proj_\f{H}(L)$
by Lemma \ref{transc}.  If $U = L$, then $L/A \in \f{H}$.  Either $L \in \f{H}$ or $L \not\in \f{H}$ and $\Proj_\f{H}(L) \ne \emptyset$ by hypothesis. \end{proof}

\begin{Cor}\label{split}  Let $\f{F}$ be a  formation.  Then $\f{F}$ is a  Schunck class if and only if $\f{F}$ is saturated. \end{Cor}

\begin{proof} Suppose $\f{F}$ is a Schunck class.  Suppose $L/\Psi(L) \in \f{F}$.  We have to
prove that $L \in \f{F}$.  We may suppose $\Psi(L) \ne 0$.  Let $A \subseteq \Psi(L)$ be a minimal ideal of $L$.  By induction, $L/A \in \f{F}$.  Let $U \in \Proj_\f{F}(L)$.  If $L \not\in \f{F}$, then by Lemma \ref{Hcomp}, $U$ complements $A$ in $L$ contrary to $A \subseteq \Psi(L)$.  Therefore    $L
\in \f{F}$.

Suppose $\f{F}$ is saturated.  Suppose $L \not\in \f{F}$ and that $A$ is a minimal ideal of $L$ with $L/A \in \f{F}$.  By Lemma \ref{schunck}, it is sufficient to prove that $\Proj_\f{F}(L) \ne \emptyset$.  Since $\f{F}$ is saturated, $A \not\subseteq \Psi(L)$, so there exists a maximal
subalgebra $U$ complementing $A$ in $L$.  Suppose $K$ is a  ideal of $L$ and that $U+K/K$ is
not $\f{F}$-maximal in $L/K$.  Then we must have $K \subseteq U$ and $L/K \in \f{F}$.  But then
$L/(K \cap A) \in \f{F}$ since $\f{F}$ is a formation, that is, $L \in \f{F}$.
\end{proof}

\begin{lemma}\label{satH} Let $\f{H}$ be a  Schunck class.  Then
\begin{enumerate}
\item $L,M \in \f{H}$ implies $L \oplus M \in \f{H}$, and 
\item $L/ \Psi(L) \in \f{H}$ implies $L \in \f{H}$.\qed \end{enumerate}
\end{lemma}

\begin{Cor}\label{central}   Let $\f{H} \ne \{0\}$ be a  Schunck class.  Suppose $A$ is a central ideal of $L$ and $L/A \in \f{H}$.  Then $L \in \f{H}$. \qed\end{Cor}

\section{Schunck formations}
In this section, I consider the special properties of those Schunck classes which are formations.

\begin{definition}  A homomorph $\f{H}$ is said to be \textit{split} if $L \in \f{H}$ and $A$ an abelian ideal of $L$ imply that the split extension of $A$ by $L/A$ is
also in $\f{H}$. \end{definition}

\begin{lemma}\label{Fissplit} Let $\f{F}$ be a formation.  Then $\f{F}$ is split. \qed
\end{lemma}

\begin{lemma}\label{SplitF} Let $\f{H}$ be a Schunck class which is split.  Then $\f{H}$ is a formation. \qed\end{lemma}

\begin{definition}  Let $\f{F}$ be a formation of soluble Leibniz algebras.  The class {\em locally defined by} $\fF$ is the class $\loc(\fF)$ of all Leibniz algebras $L$ with the property that, for each chief factor $A/B$ of $L$, we have $L/ \cser_L(A/B) \in \fF$. \end{definition}

Since $N(L)$ is the intersection of the centralisers of the chief factors of $L$, this condition is equivalent to $L/N(L) \in \fF$.

\begin{theorem}  Let $\fF$ be a formation of soluble Leibniz algebras.  The the class $\f{H}$ locally defined by $\fF$ is a Schunck formation.
\end{theorem} 

\begin{proof}  Since the condition for an algebra $L$ to be in $\f{H}$ is a condition on the chief factors, $\f{H}$ is a formation.  We have to prove that it is saturated.  So suppuse that $K \id L$, $K \subseteq \Phi(L)$ and that $L/K \in \fH$.  Let $N/K = N(L/K)$.  Then  $N$ is nilpotent by Barnes \cite[Theorem 3.6]{LeibEngel}, so $N = N(L)$ and $L \in \fH$.
\end{proof}

Note that, as $L/\cser_L(A/B)$ is a Lie algebra, we can always replace the defining formation $\fF$ by the class of Lie algebras in $\fF$.  The class $\f{N}$ of nilpotent algebras is locally defined by the class $\{0\}$.  We denote by $\f{Cs}$ the class of completely soluble Leibniz algebras, that is, those algebras with nilpotent derived algebra.  As $L'$ nilpotent is equivalent to $L/N(L)$ abelian, $\fCs$ is locally defined by the formation $\f{A}$ of abelian algebras.  By Ayupov and Omirov \cite[Theorem 4]{AyO}, a soluble Leibniz algebra over a field of characteristic $0$ is completely soluble. 

\begin{definition}  Let $\Lambda$ be a normal $F$-subspace of the algebraic closure $\bar{F}$ of the field $F$.  The {\em eigenvalue defined class} $\ev(\Lambda)$ is the class of soluble Leibniz algebras with the property that for all $x \in L$, all eigenvalues of the left and right actions of $x$ on $L$ are in $\Lambda$.
\end{definition}

Note that an eigenvalue of the action on $L$ is necessarily an eigenvalue of the action on some chief factor of $L$.  Since the chief factor is either symmetric or antisymmetric, we need only consider the left actions.  Again, the defining property is a condition on the chief factors, so $\ev(\Lambda)$ is clearly a formation. 

Let $\Lambda$ be a normal subspace of the algebraic closure $\bar{F}$ of $F$.  Let $\phi: V \to V$ be a linear transformation of the vector space $V$.  Let $f(t)$ be the characteristic polynomial of $\phi$ and let $g(t)$ be the product of the linear factors $(t-\lambda_i)$ over $\bar{F}$ with $\lambda_i \in \Lambda$.  Put $K_{\Lambda}(\phi) = \{v \in V \mid g(\phi)v = 0\}$.  Then $g(t)$ is a polynomial over $F$ and $K_{\Lambda}(\phi)$ is a subspace of dimension $\deg(g)$.  If $h(t)$ is a polynomial over $F$ all of whose roots lie in $\Lambda$ and $v \in V$ satisfies $h(\phi)v = 0$, then $v \in K_{\Lambda}(\phi)$.  

\begin{lemma} \label{K-Lambda} Suppose that $d$ is a derivation of the algebra $L$ over $F$ and that $\Lambda$ is a normal subspace of $\bar{F}$.  Then $K_{\Lambda}(d)$ is a subalgebra of $L$.
\end{lemma}

\begin{proof}  Let $\lambda_1, \dots, \lambda_n$ be the eigenvalues of $d$ which lie in $\Lambda$. Put 
$$h(t) = \Pi_{i,j}(t-\lambda_i -\lambda_j).$$ 
Suppose $a,b \in K_{\Lambda}(d)$.  Since $d(ab) = (da)b + a(db)$, we have that $h(d)(ab)=0$.  But $h(t)$ is a polynomial over $F$ all of whose roots lie in $\Lambda$.  Therefore $ab \in K_{\Lambda}(d)$.
\end{proof}

\begin{theorem} \label{Ev} Let $\Lambda$ be a normal subspace of $\bar{F}$.  The class $\ev(\Lambda)$ of soluble Leibniz algebras is a Schunck formation.
\end{theorem}

\begin{proof}  $\ev(\Lambda)$ is a formation.  Suppose that $A$ is a minimal ideal of $L$ and that $L/A \in \ev(\Lambda)$.  Suppose that $A \subseteq \Phi(L)$.  For $x \in L$, consider the inner derivation $d_x: L \to L$ given by $d_x(y) = xy$.  Since $L/A \in \ev(\Lambda)$, $K_{\Lambda}(d_x) + A = L$.  But $A \subseteq \Phi(L)$.  Therefore $K_{\Lambda}(d_x) = L$.  Therefore $L \in \ev(\Lambda)$.
\end{proof}

\begin{theorem} \label{th-supersat} The class $\f{U}$ of supersoluble Leibniz algebras is the Schunck formation $\f{Cs} \cap \ev(F)$.
\end{theorem}

\begin{proof} We have $\f{U} \subseteq \f{Cs} \cap \ev(F)$.  Suppose that $L \in \f{Cs} \cap \ev(F)$.  Let $A \subseteq L'$ be a minimal ideal of $L$.  By induction, we may suppose that $L/A$ is supersoluble.  We have to prove $\dim(A)=1$.  But $A$ is an irreducible module for the abelian Lie algebra $L/L'$ and all eigenvalues of the left actions of elements are in $F$.  Therefore there exists a $1$-dimensional invariant subspace.  As $A$ is either symmetric or antisymmetric, this subspace is also invariant under the right actions, so $\dim(A) = 1$.
\end{proof}

\begin{cor} Let $L$ be a soluble Leibniz algebra.  Then $L$ is supersoluble if and only if every maximal subalgebra of $L$ has codimension $1$. \end{cor}

\begin{proof}  We use induction over $\dim(L)$.  Suppose $L$ is supersoluble and let $A$ be a minimal ideal of $L$.  Suppose $M$ is a maximal subalgebra of $L$.  If $M \supseteq A$, then by induction. $\codim(M) = 1$.  If $M \not\supseteq A$, then $M$ complements $A$ in $L$ and $\codim(M) = \dim(A) = 1$. 

Now suppose that every maximal subalgebra has codimension $1$.  Then $L/A$ is supersoluble.  If $L$ is not supersoluble, then by Theorem \ref{th-supersat}, there exists a $\fU$-projector $M$ of $L$.  As $M$ complements $A$, $\dim(A) =\codim(M) =  1$.  \end{proof}

\begin{definition} Let $\fF$ be a Schunck formation of soluble Leibniz algebras. 
Let $V$ be an irreducible $L$-module.  Then $V$ is called $\fF$-central if the split extension of $V$ by $L/\cser_L(V)$ is in $\fF$, and is called $\fF$-eccentric otherwise.  An $L$-module $W$ is called $\fF$-hypercentral if every composition factor of $W$ if $\fF$-central, and is called $\fF$-hypereccentric if every composition factor if $\fF$-eccentric. \end{definition}

\begin{theorem}\label{th-Fchief} Let $\fF$ be a Schunck formation and let $L$ be a Leibniz
algebra.  Then $L \in \fF$ if and only if every chief factor of $L$ is an $\fF$-central $L$-bimodule. \qed\end{theorem}

Let $L$ be a Leibniz algebra and let $V,W$ be $L$-modules.  We can define a left action of $L$ on $V \otimes W$ as for Lie algebras, setting $x(v \otimes w) = xv \otimes w + v \otimes xw$.  There is in general, no such obvious way to define a right action.  If either, both $V$ and $W$ are symmetric or both are antisymmetric, then we can set $(v \otimes w)x = vx \otimes w + v \otimes wx$.  

\begin{lemma} \label{tensor} Suppose that $V, W$ are $\fF$-hypercentral $L$-modules, either both symmetric or both antisymmetric.  Then $V \otimes W$ is $\fF$-hypercentral.
\end{lemma}

\begin{proof} We need only consider the case where both $V$ and $W$ are irreducible, then general case then following.  By replacing $L$ by $L/(\cser_L(V) \cap \cser_L(W))$, we can assume that $L$ is a Lie algebra in $\fF$.  We form the class 2 nilpotent Lie algebra $N = V \oplus W \oplus (V \otimes W)$ defining $vw = v \otimes w$ for $v \in V$ and $w \in W$.  Using the actions of $L$ on $V, W, V \otimes W$ to define the multiplication makes $X = L \oplus N$ into a Leibniz algebra in either case.  We have $\Phi(N) = N'$.  If $A \id X, A \subseteq N'$ and $A/B$ is a chief factor of $X$, then $X/B$ does not split over $A/B$ since $N/B$ does not split.  It follows that $X \in \fF$ and that $V \otimes W$ is $\fF$-hypercentral.
\end{proof}

\begin{lemma} \label{hom} Let $V, W$ be $\fF$-hypercentral $L$-modules, either both symmetric or both antisymmetric.  Then $\Hom(V,W)$ is $\fF$-hypercentral.
\end{lemma}

\begin{proof}  We need only consider the case where both $V$ and $W$ are irreducible.  The result follows from Lemma \ref{tensor} as for Lie algebras.
\end{proof}

\begin{theorem}\label{th-Fcomp} Let $\fF$ be a Schunck formation of soluble Leibniz algebras and suppose $L \in \fF$.  Let $V$ be an $L$-module.  Then $V$ is the direct sum of an $\fF$-hypercentral module and an $\fF$-hypereccentric module. \end{theorem}
\begin{proof} As for Lie algebras, using the above lemmas.   See \cite[Theorem 4.4]{HyperC}\end{proof}

Let $\fF$ be a Schunck formation of soluble Leibniz algebras.  Then the class $\Lie\fF$ of Lie algebras in 
$\fF$ is a Schunck formation of soluble Lie algebras.  The primitives of $\Lie\fF$ are the primitives of $\fF$ which are Lie algebras.  If $L \in \Lie\fF$ and $V$ is an irreducible $L$-module, the split extension of $\sym V$ by $L$ is again a Lie algebra.  When I wish to emphasize that I am considering the module in the context of Lie algebras, I denote $\sym V$ by $\Lie V$.

\begin{theorem} \label{pair} Let $\fF$ be a Schunck formation of soluble Leibniz algebras.  Let $L$ be a soluble Leibniz algebra and let $V$ be an irreducible $L$-module.  Then $\sym V$ is $\fF$-central if and only if $\asym V$ is $\fF$-central.
\end{theorem}

\begin{proof}  We may suppose that $V$ is non-trivial.  By replacing $L$ with $L/\ker(V)$, we may suppose that $V$ is faithful and that $L$ is a Lie algebra.  We may suppose that at least one of $\sym V, \asym V$ is $\fF$-central, so $L \in \Lie\fF$.

 We show that there exists $a_0 \in L$ such that the left action $T_{a_0} \co V \to V$ is invertible.  Let $A$ be a minimal ideal of $L$.  For any $a \in A$, $\langle a \rangle \sn L$, so all $\langle a \rangle$-composition factors of $\Lie V$ are isomorphic.  If the action of $a$ on the composition factors is trivial, then $T_a$ is nilpotent.  If all $T_a$ are nilpotent, then by Engel's Theorem, the space $W= \{v \in V \mid av=0 \text{ for all } a \in A\} \ne 0$.  But $W$ is a submodule of the irreducible module $\Lie V$, so $W=V$ contrary to $V$ being faithful.  Therefore there exists $a_0 \in A$ such that the action of $a_0$ on the composition factors is non-trivial.  This implies that $T_{a_0}$ is invertible.

If $\dim(L) = 1$, we replace $L$ with the Lie algebra $L \oplus \langle a_1 \rangle$ and still have $L \in \Lie\fF$.  We then have the conditions for Lemma \ref{LP} which does not require that the module $V$ be faithful.  By Lemma \ref{LP}, the Loday-Pirashvili exact sequence
$$0 \to \asym V \to \bar{V} \to \Hom(L,V) \to 0$$
does not split.  

Suppose that $\sym V$ is $\fF$-central.  Then $\Hom(L, V)$ is $\fF$-hypercentral.  Since $\bar V$ does not split over $\asym V$, by Theorem \ref{th-Fcomp}, $\asym V$ is $\fF$-central.

Now suppose that $\asym V$ is $\fF$-central but that $\sym V$ is $\fF$-eccentric.  We have $L \in \Lie\fF$.  By Barnes \cite[Theorem 2.3]{excentric}, $\Hom(L, \Lie V)$ is $\Lie\fF$-hypereccentric and so is $\fF$-hypereccentric.  But $\asym V$ is $\fF$-central and $\bar V$ does not split over $\asym V$ contrary to Theorem \ref{th-Fcomp}.
\end{proof}

\begin{cor}\label{one-one} $\fF \leftrightarrow \Lie{\fF}$ is a one-to-one correspondence between the Schunck formations of soluble Leibniz algebras and the Schunck formations of soluble Lie algebras.
\end{cor}

\begin{theorem} Let $\fF$ be a Schunck formation of soluble Leibniz algebras.  Let $V, W$ be $L$-modules, either both symmetric or both antisymmetric.  Suppose that $V$ is $\fF$-hypercentral and that $W$ is $\fF$-hypereccentric.  Then $V \otimes W$ and $\Hom(V,W)$ are $\fF$-hypereccentric.
\end{theorem}

\begin{proof}  We need only consider the case where both $V$ and $W$ are irreducible.    By working with $L/(\ker V \cap \ker W)$, we may suppose that $L$ is a Lie algebra.  By Theorem \ref{pair}, we need only consider the case where both $V$ and $W$ are symmetric.  Then the split extension of $V$ by $L$ is a Lie algebra in $\Lie \fF$.  The split extension of $W$ by $L$ is a lie algebra not in $\fF$, so not in $\Lie\fF$.  By Barnes \cite[Theorem 2.3]{excentric}, $V \otimes W$ and $\Hom(V,W)$ are $\Lie\fF$-hypereccentric and it follows that $V \otimes W$ and $\Hom(V,W)$ are  $\fF$-hypereccentric.
\end{proof}

We can now strengthen Theorem \ref{th-Fcomp}.

\begin{theorem} \label{InvComps}  Let $\fF$ be a Schunck formation of soluble Leibniz algebras.  Let $L$ be a (not necessarily soluble) Leibniz algebra and suppose that $U \sn L$, $U \in \fF$.  Let $V$ be an $L$-module.  Let $V^+$ and $V^-$ be the $\fF$-hypercentral and $\fF$-hypereccentric components of $V$ as $U$-module.  Then $V^+, V^-$ are $L$-submodules of $V$.
\end{theorem}

\begin{proof}   Suppose $U=U_0 \id U_1 \id \dots \id U_n = L$.  By Theorem \ref{th-Fcomp}, the result holds if $n = 0$.  We use induction over $n$, so we suppose $V^+, V^-$ are $U_{n-1}$-submodules.  Let $W$ be either of $V^+, V^-$ and let $x \in L$.  For $u \in U$, $ux, xu \in U_{n-1}$.  Thus for $w \in W$, we have $u(xw) = (ux)w + x(uw) \in W + xW$ and $(xw)u = x(wu) - w(xu) \in xW + W$.  Thus $W+xW$ is a $U$-submodule.  The map $\phi:W \to W+xW/W$ given by $\phi(w) = xw +W$ is a $U$-module homomorphism as
\begin{equation*}\begin{split} u \phi(w) &= u(xw)+W = (ux)w +x(uw)+W =\phi(uw),\\
\phi(w)u &= (xw)u +W = x(wu) - w(xu) +W = \phi(wu),
\end{split}\end{equation*}
since $(ux)w, w(xu) \in W$.  But $W$ and $V/W$ have no $U$-bimodule composition factors in common.  Therefore $\phi = 0$ and $xW \subseteq W$.

We now consider $Wx+W$.  We have $u(wx) = (uw)x + w(ux) \in Wx+W$ and $(wx)u = w(xu) - x(wu) \in W$ since $xW \subseteq W$.  Thus $W+Wx$ is a $U$-submodule and the right action of $U$ on $W+Wx/W$ is zero.  The map $\psi: W \to Wx+W/W$ given by $\psi(w) = wx + W$ is a left $U$-module homomorphism as $u \psi(w) = u(wx) +W = (uw)x + w(ux)+W = \psi(uw)$.  Thus any $U$-module composition factor of $Wx+W/W$ is isomorphic to $\asym{X}$ for some $U$-module composition factor $X$ of $W$.  But $\asym{X}$ is $\fF$-central if and only if $X$ is $\fF$-central.  Since $W$ is one of the components $V^+, V^-$, it follows that $W+Wx/W = 0$ and $Wx \subseteq W$.
\end{proof}

\section{$\fF$-normalisers}
Let $\fF$ be a Schunck formation of soluble Leibniz algebras.  We cannot follow the definition used in the theory of soluble groups as the required structures do not exist in Leibniz or Lie algebras.  In \cite{Stitz-CA}, Stitzinger defined $\fF$-normalisers for Lie algebras using a property proved for those of groups.  I copy that approach with only minor changes arising from the need to use $\Psi(L)$ in place of $\Phi(L)$ at some points.  The following two results are essentially Proposition 10 and Theorem 3 of Stitzinger \cite{Stitz-Frat}.  Note that assuming $\Psi(L)=0$ is weaker than assuming $\Phi(L)=0$.

\begin{lemma} \label{abcomp}Let $L$ be a soluble Leibniz algebra with $\Psi(L)=0$.  Let $A$ be an abelian ideal of $L$.  Then $A$ is complemented in $L$.
\end{lemma}

\begin{proof}  There exists a maximal subalgebra $M$ of $L$ which does not contain $A$.  Then $B = A\cap M \id L$. If $A$ is a minimal ideal, the result holds.  So suppose that $B \ne 0$.  We use induction over $\dim(A)$.  By induction, there exists a subalgebra $C$ which complements $B$.  Put $D = M \cap C$.  Then 
$$D \cap A = (C \cap M )\cap A = C \cap B = 0.$$
By the modular law for subspaces, since $B \subseteq M$, we have $(B+C)\cap M = B+(C\cap M)$.  Thus
$$M = L \cap M = (B+C) \cap M = B+(C\cap M) = B+D$$
and $A + D = A + (B+D) =A+M = L$.
\end{proof}

\begin{theorem}\label{psi0} Let $L$ be a soluble Leibniz algebra with $\Psi(L)=0$.  Then
$$\soc(L) = \cser_L(\soc(L)) = N(L)
$$ and is complemented in $L$.
\end{theorem}

\begin{proof}  Put $S = \soc(L)$.  By Lemma \ref{abcomp}, there exists a complement $C$ to $S$.  Put $D = C \cap \cser_L(S)$.  Then $D \id L$.  If $D \ne 0$, then $D$ contains a minimal ideal of $L$ contrary to $C \cap \soc(L) = 0$.  Therefore $\cser_L(S) = S$.  But $N(L)$ centralises all chief factors, so $N(L) \subseteq \cser_L(S)$.
\end{proof}

In the following, $\fF$ is a Schunck formation of soluble Leibniz algebras.

\begin{definition}    A maximal subalgebra $M$ of a soluble Leibniz algebra is called $\fF$-{\em normal} in $L$ if $L/\core_L(M) \in \fF$, and $\fF$-{\em abnormal} otherwise.
\end{definition}

\begin{lemma} A maximal subalgebra $M$ of $L$ is $\fF$-normal if and only if it complements an $\fF$-central chief factor of $L$. \qed
\end{lemma}

\begin{lemma}  $L \in \fF$ if and only if every minimal ideal of $L/ \Psi(L)$ is $\fF$-central. \qed
\end{lemma}

\begin{definition}  A maximal subalgebra $M$ of $L$ is called $\fF$-{\em critical} in $L$ if $M$ is $\fF$-abnormal and $M + N(L) = L$. 
\end{definition}
\begin{definition}  A chief factor $U/V$ of $L$ is called $\fF$-{\em critical} if it is $\fF$-eccentric and all factors below it are either $\fF$-central or uncomplemented.
\end{definition}

\begin{lemma} Let $M$ be a maximal subalgebra of $L$.  Then $M$ is $\fF$-critical in $L$ if and only if $M$ complements an $\fF$-critical chief factor of $L$. \qed
\end{lemma}

\begin{theorem} $L$ has an $\fF$-critical maximal subalgebra if and only if $L \notin \fF$. \qed
\end{theorem}

\begin{definition}  A subalgebra $U$ of $L$ is called an $\fF$-{\em normaliser} of $L$ if there exists a chain $U = U_0 \subset U_1 \subset \dots \subset U_n = L$ such that each $U_i $ is an $\fF$-critical maximal subalgebra of $U_{i+1}$ and $U$ has no $\fF$-critical maximal subalgebras.
\end{definition}

It is clear that every soluble Leibniz algebra $L$ has an $\fF$-normaliser and that if $U$ is an $\fF$-normaliser of $L$, then $U \in \fF$.

\begin{theorem}  Let $K$ be an $\fF$-normaliser of $L$.  Then $K$ covers every $\fF$-central chief factor of $L$ and avoids every $\fF$-eccentric chief factor of $L$. \qed
\end{theorem}

There are many more results in Stitzinger's paper \cite{Stitz-CA} which generalise to Leibniz algebras.  I quote here only those needed later in this paper.

\section{Frattini properties}

By Barnes and Newell \cite[Theorem 4.3]{BN}, Schunck classes of soluble Lie algebras are $\fratsn$-closed.  This does not hold for Schunck classes of soluble Leibniz algebras in general.

\begin{example} \label{notstrongcl} Let $\f{U}_1$ be the Schunck class of soluble Leibniz algebras whose primitives are the $1$-dimensional Lie algebra and the non-abelian $2$-dimensional Lie algebra $P_2$.  Then $\f{U}_1$ is not $\fratn$-closed.
\end{example}

\begin{proof}  Let $S = \langle s_1, s_2, s_3 \rangle$ be the Lie algebra with the multiplication $s_1s_2 = s_3$, $s_2s_3=s_1$ and $s_3s_1 = s_2$.  Let $Z = \langle z \rangle$ and let $L$ be the Lie algebra $S \oplus Z$.  Let $V = \langle v_1,v_2, v_3 \rangle$ be the $L$-module giving the adjoint representation of $S$, $s_iv_i = 0$, $s_1v_2 = v_3 = -s_2v_1$, $s_2v_3 = v_1 = -s_3v_2$ and $s_3v_1 = v_2 = -s_1v_3 $ with $zv=v$ for all $v \in V$.  We regard $\Hom(S,V)$ as the $L$-module $\{f \in \Hom(L,V) \mid f(Z) = 0\}$ and set $W = \Hom(S,V) \oplus V$as left $L$-module and make it into a Leibniz module by setting $(f,v)x = (-xf, f(x))$ for all $x \in L$.  Note that $zf = f$ for all $f  \in \Hom(S,V)$.  The argument of Lemma \ref{LP} shows that the exact sequence  $$0 \to \asym V \to W \to \Hom(S,V) \to 0$$ does not split.

Now let $X$ be the split extension of $W$ by $L$.  Then $\asym V$ is a minimal ideal of $X$.  Since $W$ does not split over $\asym V$, $X$ does not split over $\asym V$ and so $\asym V \subseteq \Phi(X)$.  Put $U = W + Z$.  Then $U \id X$.  As $U/\asym V$ is a Lie algebra with the left action of $z$ on $W/\asym V$ the identity, we have $U/\asym V \in \f{U}_1$.  But put $K = \Hom(S,V) + \langle v _1, v_2 \rangle \subset W$.  Then $K \id U$ and $U/K \simeq \asym P_2$.  Thus $U \notin \f{U}_1$.
\end{proof}

\begin{theorem} \label{th-FE2star}  Suppose that $\fF$ is a Schunck formation of soluble Leibniz algebras.  Then $\fF$ is $\fratsn$-closed.
\end{theorem}

\begin{proof} Suppose that $V \id U \sn L$, $V \le \Phi(L)$ and that $U/V \in \fF$.  We use induction over $\dim(L)$.  By Barnes \cite[Theorem 3.6]{LeibEngel}, the result holds if $U/V$ is nilpotent, so we may suppose that $U_{\fN} \ne 0$.  By Barnes \cite[Theorem 2.5]{LeibSub}, $U_\fN \id L$, so there exists a minimal ideal $K$ of $L$, $K \le U$.  By induction, we have $U/K \in \fF$.  But $\fF$ is a formation, so $U/K \cap V \in \fF$.  If $K \cap V \ne K$, then $K$ as $U/K$-bimodule has an $\fF$-central composition factor.  By Barnes \cite[Theorem 2.2]{LeibSub}, all composition factors of $K$ are isomorphic, so $K$ is $\fF$-hypercentral and $U \in \fF$.  Hence we may suppose that $K \cap V = K$ and so $K \le \Phi(L)$.

As we are supposing that $U$ is not nilpotent, we have $U_\fN > K$ and there exists an ideal $A$ of $L$, $K < A \le U_\fN$,  such that $A/K$ is a minimal ideal of $L/K$.  Since $A/K$ is nilpotent and $K \le \Phi(L)$, $A$ is nilpotent.  But $A^2 \id L$, so $A^2 = K$ or $A^2 = 0$.  Now $A^2=K$ implies that $K$ is a quotient of the $\fF$-hypercentral $U$-bimodule $A/K \otimes A/K$ and so, that $U \in \fF$.  Hence we may suppose that $A^2 = 0$.  Consider $A$ as $L/K$-bimodule.  By Theorem  \ref{InvComps}, the $(U,\fF)$-hypercentral component $A^+ \id L$.  But $A^+ \ne 0$ since $A/K$ is $(U,\fF)$-hypercentral.  By induction over $\dim(L)$, $U/A^+ \in \fF$ and as $A^+$ is $(U,\fF)$-hypercentral, we have $U \in \fF$. 
\end{proof}

There are some special cases for which the stronger result that $\fF$ is $\fratsnr$-closed holds.

\begin{theorem}  Suppose that $\fF = \loc(\f{K})$.  Then $\fF$ is $\fratsnr$-closed.
\end{theorem}

\begin{proof}  Suppose that $V \id U \snr L$, $V \le \Phi(L)$ and that $U/V \in \fF$.   Let $N/V$ be the nil radical of $U/V$.  Then $N \snr L$ and by Barnes \cite[Theorem 3.6]{LeibEngel}, $N$ is nilpotent.  Since $U/N \in \fK$, we have $U \in \loc(\f{K})$.
\end{proof}

\begin{theorem} \label{evphi}  Let $\fF = \ev(\Lambda)$ be an eigenvalue defined formation.  Then $\fF$ is $\fratsnr$-closed.
\end{theorem}

\begin{proof}  Suppose $V \ideq U = U_0 \idr U_1 \idr \dots \idr U_n = L$,  $V \subseteq \Phi(L)$ and  $U/V \in \fF$.  Take $u \in U$ and let $d_u\co L \to L$ be the inner derivation $d_u(x) = ux$.  Then $d^n_u(L) \subseteq U$.  Since $0 \in \Lambda$ and $U/V \in \ev(\Lambda)$, $K_{\Lambda}(d_u) + V = L$.  But $V \subseteq \Phi(L)$.  Therefore $K_{\Lambda}(d_u) = L$. Since for all $u \in U$, all eigenvalues of $d_u|_U$ are in $\Lambda$, $U \in \ev(\Lambda)$.
\end{proof}

\begin{cor}  Suppose $\ch(F)=0$.  Let $\fF$ be a Schunck formation.  Then $\fF$ is $\fratsnr$-closed.
\end{cor}

\begin{proof}  By Barnes \cite{SatF0}, $\Lie\fF$ is eigenvalue defined, and it follows from Corollary \ref{one-one}, that $\fF$ is eigenvalue defined.
\end{proof}

\section{Intravariance}

A subalgebra $U$ of a Lie algebra $L$ is called intravariant in $L$ if every derivation of $L$ is the sum of an inner derivation and a derivation whch stabilisees $U$.
This is not appropriate for Leibniz algebras as it ignores the right actions of elements.  Loday in \cite{Lod}, has introduced the concept of a biderivation which remedies this deficiency.

\begin{definition}   Let $A$ be an algebra.  An {\em antiderivation} of $A$ is a linear map $D \co A \to A$ such that $D(xy) = xDy - yDx$. A {\em biderivation} is a pair $(d,D)$ where $d$ is a derivation and $D$ is an antiderivation such that $(dx)y = (Dx)y$ for all $x,y \in A$.
\end{definition}

\begin{definition} Let $L$ be a Leibniz algebra.  The {\em inner biderivation} given by $a \in L$ is the pair $\Biad_a = (\ad_a, \Ad_a)$, where $\ad_a(x) = ax$ and $\Ad_a(x) = -xa$.
\end{definition}

The biderivations of an algebra $A$ with multiplication
$$[(d,D), (d',D')] = (dd'-d'd, dD' - D'd)$$
form a Leibniz algebra denoted by $\Bider(A)$.  If $L$ is a Leibniz algebra and $a \in L$, then $\Biad_a$ is a biderivation of $L$ and the map $a \mapsto \Biad_a$ is a Leibniz algebra homomorphism of $L$ into $\Bider(L)$.  

Unfortunately, biderivations are too general for the purposes of this section.
Given any algebra $A$ and an algebra $B$ of biderivations of $A$, we can form the split extension $X = A+B$ of $A$ by $B$ by defining the products of $a \in A$ and $b = (d,D) \in B$ by $[b,a] = da$ and $[a,b] = -Da$.  For the split extension to be a Leibniz algebra, (as well as needing $A$ itself to be a Leibniz algebra,) we have to restrict the biderivations in $B$.

\begin{definition}  Let $L$ be a Leibniz algebra. Let $d$ be a derivation and $D$ an antiderivation of $L$.  We say that the pair $(d,D)$ is {\em close} if $(d-D)L \subseteq \Leib(L)$.
\end{definition} 

A close pair is a biderivation since $(d-D)x \in \Leib(L)$ implies $((d-D)x)y = 0$.    An inner biderivation $\Biad_a$ is close as $\ad_a(x)-\Ad_a(x) = ax + xa \in \Leib(L)$.

\begin{lemma} The close biderivations of a Leibniz algebra $L$ form a subalgebra $\Cloder(L)$ of the biderivation algebra $\Bider(L)$.
\end{lemma}

\begin{proof}  For any derivation $d$ of $L$, we have $dx^2 = (dx)x + x(dx) \in \Leib(L)$.  Thus $d\Leib(L) \subseteq \Leib(L)$.  For any antiderivation $D$ of $L$, we have $Dx^2 = x(Dx) -x(Dx) = 0$  Thus $D\Leib(L) = 0$.  

Suppose that $b = (d,D)$ and $b'=(d',D')$ are close.  We have to prove that $(dd'-d'd, dD'-D'd)$ is close.  For $x \in L$, we have
$$((dd'-d'd)-(dD'-D'd))x = d(d'-D')x - (d'-D')dx.$$
Since $(d',D')$ is close, $(d'-D')x$ and $(d'-D')dx$ are in $\Leib(L)$ and it follows that $[b,b']$ is close.
\end{proof}

\begin{lemma}  Suppose $B$ is an algebra of close biderivations of the Leibniz algebra $L$.  Then the split extension of $L$ by $B$ is a Leibniz algebra.
\end{lemma}

\begin{proof} The Leibniz identity for the triples $b(xy)$, $x(by)$, $x(yb)$, $b(b'x)$ and $b(xb')$ where $x,y \in L$ and $b = (d,D), b'=(d',D') \in B$ is easily verified.
Consider $x(bb') = -dD'x + D'dx$.  Since $(d,D)$ is close, we have $(d-D)x \in \Leib(L)$ and $D'(d-D)x = 0$.  Thus $x(bb') = -dD'x + D'Dx$.  But $(xb)b' = (-Dx)b' = D'Dx$ and $b(xb') = b(-D'x) = -dD'x$.  Thus the Leibniz identity holds.
\end{proof}

\begin{definition}  A subalgebra $U$ of the Leibniz algebra $L$ is called {\em intravariant} in $L$ if every close biderivation of $L$ is the sum of an inner biderivation and a biderivation which stabilises $U$.
\end{definition}

If $L$ is a Lie algebra, then a close pair $(d,D)$ must have $d=D$.  Thus the algebra $\Cloder(L)$ is essentially just the derivation algebra of $L$ and the intravariant subalgebras of $L$ considered as Leibniz algebra are precisely the subalgebras which are intravariant under the definition for Lie algebras.

\begin{theorem} \label{fratarg} Let $L$ be a Leibniz algebra and let $U$ be a subalgebra of $L$.  Then $U$ is intravariant in $L$ if and only if, whenever $L$ is an ideal of a Leibniz algebra $X$, we have $X = L + \nser_X(U)$.
\end{theorem}

\begin{proof}  Suppose that $L \id X$ and that $U$ is intravariant in $L$.  Let $x \in X$.  Then $\Biad_x|_L$ is a close biderivation of $L$, so we have $a \in L$ such that $\Biad_x - \Biad_a$ is a biderivation which stabilises $U$.  Put $n = x-a$.  Then $\Biad_n$ stabilises $U$, so $n \in \nser_X(U)$ and $x \in L+\nser_X(U)$.

Suppose conversely, that for every Leibniz algebra $X$ which contains $L$ as an ideal, we have $X = L + \nser_X(U)$.  Let $b=(d,D)$ be a close biderivation of $L$ and let $B$ be the subalgebra of $\Cloder(L)$ generated by $b$.  Form the split extension $X$ of $L$ by $B$.  Then $X = L + \nser_X(U)$, so $b = a + n$ for some $a \in L$ and $n \in \nser_X(U)$.  We have $\Biad_b = \Biad_a + \Biad_n$.  But $b = \Biad_b|_L$ and the result follows.
\end{proof}

 In Barnes \cite{Frat}, it was shown that for Schunck classes $\f{H}$ of soluble Lie algebras, all $\f{H}$-projectors are intravariant.  That this does not hold for Leibniz algebras follows from Example \ref{notstrongcl} and the following theorem.

\begin{theorem} \label{intra-strong} Let $\fH$ be a Schunck class of soluble Leibniz algebras.  Then $\fH$ is $\fratn$-closed if and only if all $\fH$-projectors are intravariant. 
\end{theorem}

\begin{proof}  Suppose $\fH$ is $\fratn$-closed.  I use induction over $\dim(L)$ to prove that if $N$ is a soluble ideal of $L$ and $U$ is an $\fH$-projector of $N$, then $N+ \nser_L(U) = L$.  Let $K \le N$ be a minimal ideal of $L$ and let $M = \nser_L(K+U)$.  By induction, $M+N = L$.  Suppose $M < L$.  Then $U$ is an $\fH$-projector of $M \cap N$ and by induction, $\nser_M(U) + (N\cap M) = M$, whence it follows that $N+\nser_L(U) = L$.  Thus we may suppose that $U+K \id L$ which implies $U+K = N$ by Lemma \ref{selfnorm}.  If $K \le \Phi(L)$, then $N \in \fH$ since $\fH$ is $\fratn$-closed.  In this case, we have $\nser_L(U) = L$.  If $K \not\le \Phi(L)$, the there exists a maximal subalgebra $M$ of $L$ which complements $K$ in $L$.  Then $V = M \cap N$ complements $K$ in $N$.  Since $N \id L$, by Barnes \cite[Theorem 2.2]{LeibSub}, as $N$-module and hence as $U$-module, all composition factors of $K$ are isomorphic.  Further, all are complemented.  It follows that either the split extension of a composition factor by $V$ is in $\f{H}$ and $U+K \in \fH$, or $V$ is an $\f{H}$-projector of $N$, in which case, by Lemma \ref{conj}, there exists $k \in K$ such that $\alpha_k(V) = U$.  By replacing $M$ by $\alpha_k(M)$, we may assume that $M \cap N = U$.  But then $U \ideq M$ and $M = \nser_L(U)$.

Suppose conversely, that $\fH$ has the property that $\fH$-projectors are intravariant.  Let $L$ be a Leibniz algebra, $B \ideq A \ideq L$, $B \subseteq \Phi(L)$ and that $A/B \in \fH$.  By Barnes \cite[Theorem 2.5]{LeibSub},  $B_{\fN} \ideq L$. Since $B_{\fN} \subseteq \Phi(L)$, $B_{\fN}$ is nilpotent and so $A$ is soluble.  Let $U$ be an $\fH$-projector of $A$.  Then $L = A + \nser_L(U)$.  But $B + U = A$ and $B + \nser_L(U) = B+U + \nser_L(U) = L$.  Since $B \le \Phi(L)$, this implies that $\nser_L(U) = L$.  By Lemma \ref{selfnorm}, $U = A$ and $A \in \f{H}$.
\end{proof}

The above proof of the equivalence of $\fratn$-closure and the intravariance property remains valid if we restrict the larger algebra for each property  to being soluble, that is, if we work within the category of soluble Leibniz algebras.  However, if $\fH$ has the intravariance property within the category of soluble Leibniz algebras, then it has the property in the larger category.  For suppose $x \in L$.  Then the algebra $L_x = \alg x, A\rangle$ is soluble and so $L_x = A + \nser_{L_x}(U)$, that is, $x \in A+\nser_L(U)$ and it follows that $L= A + \nser_L(U)$.  It follows that if $\f{H}$ has the $\fratn$-closure property within the category of soluble Leibniz algebras, it also has that property in ther larger category.

\begin{theorem}  Let $\fF$ be a Schunck formation of soluble Leibniz algebras.  Let $L$ be a soluble Leibniz algebra and let $U$ be an $\fF$-normaliser of $L$.  Then $U$ is intravariant in $L$.
\end{theorem}

\begin{proof}  Suppose that $L$ is an ideal of the Leibniz algebra $X$.  I use induction over $\dim(L)$ to prove $L+\nser_X(U) = X$.  Let $A \subseteq L$ be a minimal ideal of $X$.  Then $U+A/A$ is an $\fF$-normaliser of $L/A$.  By induction, $L + \nser_X(U+A) = X$.  Put $N = \nser_X(U+A)$.  Since $A$ is an irreducible $X$-module, by Barnes \cite[Theorem 2.2]{LeibSub}, all $L$-composition factors of $A$ are isomorphic, so either all are $\fF$-central or all are $\fF$-eccentric.  If all are $\fF$-central, then $U \supseteq A$, $\nser_X(U)=N$ and the result holds.  If all are $\fF$-eccentric, then $U \cap A = 0$ and $U$ is an $\fF$-projector of $U+A$ and so is intravariant in $U+A$.  As $U+A \id N$, this implies that $(U+A) +\nser_N(U) = N$ and $\nser_X(U) + L \supseteq \nser_N(U) + (U+A) + L = N+L = X$.
\end{proof}

\bibliographystyle{amsplain}

\end{document}